\documentclass[11pt]{amsart}

\usepackage[margin=1in]{geometry}

\usepackage{amsmath}
\usepackage{mathtools}
\usepackage{amssymb}
\usepackage{amsfonts}
\usepackage{tikz}
\usepackage{quiver}
\usepackage[framemethod=tikz]{mdframed}
\mdfsetup{linecolor=black!100, roundcorner=3pt}

\usepackage[shortlabels]{enumitem}

\usepackage[english]{babel}
\usepackage{amsthm}

\usepackage{csquotes}

\usepackage[pdftex,colorlinks,citecolor=blue]{hyperref}

\usepackage{cleveref}

\usepackage[backend=biber, style=alphabetic, sorting=anyt, backref=page, maxnames=5, url=false]{biblatex}

\newcommand{\newword}[1]{\textit{#1}}

\theoremstyle{plain}
\newtheorem{theorem}{Theorem}[section]
\newtheorem{lemma}[theorem]{Lemma}
\newtheorem*{utheorem}{Theorem}
\newtheorem{proposition}[theorem]{Proposition}
\newtheorem{corollary}[theorem]{Corollary}

\theoremstyle{definition}

\newtheorem{question}[theorem]{Question}

\newtheorem{remark}[theorem]{Remark}
\newtheorem*{ack}{Acknowledgements}
\newtheorem*{comparison}{Comparison to Other Works}

\everymath{\displaystyle}

\renewcommand{\P}{\mathbb P}
\newcommand{\Q}{\mathbb Q}

\newcommand{\Z}{\mathbb Z}

\DeclarePairedDelimiter{\floor}{\lfloor}{\rfloor}
\DeclarePairedDelimiter{\ceil}{\lceil}{\rceil}

\title[Realizability in the Orthogonal Grassmannian]{Realizability of Cohomology Classes in Orthogonal Grassmannians}
\author{Naima Nader}
\address{University of Illinois Chicago}
\email{nnade@uic.edu}

\begin{document}

\begin{abstract}
Building on the work of Coskun and Ross, we investigate which cohomology classes 
of the orthogonal Grassmannian $OG(k,n)$ can be realized as irreducible subvarieties 
of $OG(k,n)$. We give a criterion for when the realizable classes in the orthogonal 
Grassmannian agree with those in the Grassmannian. We give a characterization for 
realizability in low dimension and codimension and give necessary conditions for 
realizability in certain orthogonal Grassmannians. 
\end{abstract}

\maketitle

\section{Introduction}
Let $X$ be a smooth complex projective variety. 
Given $Y \subset X$ a subvariety, we define the cohomology class $[Y]$ to be the Poincar\'e dual of the fundamental class of $Y$.
Following \cite[Definition 1.5]{HHMWW25}, we say that a cohomology class $\nu\in H^{2m}(X,\Z)$ is \newword{realizable over} $\Z$ if there exists a subvariety $Y\subseteq X$ with $\nu = [Y]$. 
A class $\nu\in H^{2m}(X,\Q)$ is \newword{realizable over} $\Q$ if $\lambda \nu$ is realizable over $\Z$ for some $\lambda\in \Q^+$.

In this paper, we address the question of realizability of cohomology classes of the orthogonal Grassmannian.
Specifically, let $V$ be an $n$-dimensional complex vector space and $Q_0$ a non-degenerate symmetric bilinear form on $V$.
For $2k \leq n$, we denote by $OG(k,n)$ the orthogonal Grassmannian parameterizing isotropic $k$-planes in $V$ with respect to $Q_0$. 
\begin{question}\label[question]{question1}
When is a cohomology class $\nu\in H^{2m}(OG(k,n),\Z)$ the class of an irreducible subvariety of $OG(k,n)$?
When is $\nu\in H^{2m}(OG(k,n),\Q)$ realizable over $\Q$?
\end{question}

We show that for fixed $m$ and $k$, when $n$ is large enough, the only realizable classes in $OG(k,n)$ are those coming from the Grassmannian, with explicit bounds given in \Cref{prop:iff}.
For sufficiently small dimension and codimension this reduces questions of realizability from $OG(k,n)$ to $G(k,n)$, which we resolve using \cite{CR26}.

For example, under the conditions $k>2$ and $n>2k+4$, we will see that every effective cohomology class in $H^{2N-4}(OG(k,n),\Z)$ is realizable as an irreducible surface (where $N=\dim OG(k,n)$). 
Note that, whenever $X$ has dimension $N$ and $Y \subseteq X$ dimension $r$, the cohomology class $[Y]$ belongs to $H^{2N-2r}(X, \Z)$.
For notational convenience, in expressions of this form, $N$ will always denote the dimension of the ambient variety.

Schubert classes form an additive basis for the cohomology of $OG(k,n)$, and are parameterized by $k$-tuples $(a_\bullet;b_\bullet)$, as explained in \Cref{sec:notation}. 
Any effective class is a nonnegative linear combination of Schubert classes \cite[Proposition 2.20]{C18}. 
When $k=2$ we prove the following:

\begin{utheorem}[\Cref{thm:surfaces-OG2n}]
Let $\nu=a\sigma_{1,4;} + b\sigma_{2,3;}$ be a non-zero integral cohomology class in $H^{2N-4}(OG(2,n),\Z)$, with $n\ge 9$. 
Then $\nu$ is realizable over $\Z$ if and only if $a>0$ and $b\ge 0$ or $(a,b)=(0,1)$.
\end{utheorem}

In codimension 2, we classify all realizable classes when $k>2$ and $n>2k+4$ and when $k=2$ and $n>8$ (\Cref{thm:codim2}). 
When the Picard rank of $OG(k,n)$ is 2 we prove:

\begin{utheorem}[\Cref{thm:pic2}]
Let $\nu=a\sigma_{1,2,\ldots,k-2,k,k+1;}+b\sigma_{1,2,\ldots,k-1;k-1}+c\sigma_{1,2,\ldots,k-2,k;k}$ be a non-zero integral cohomology class in $H^{2N-4}(OG(k,2(k+1)),\Q)$.
If $\nu$ is realizable over $\Q$, then $a,b,c\ge 0$ and $b^2\ge ac$. 
Conversely, if $a,b,c\ge 0$, $a+c\ge b$ and $b^2\ge ac$, then $\nu$ is realizable over $\Q$.
\end{utheorem}

\begin{remark}
Note that $OG(1,n)$ is isomorphic to a quadric hypersurface $Q\subset \P^{n-1}$.
The Schubert varieties are linear spaces on $Q$ and singular quadrics. 
Every effective class is realizable with the following exceptions: $a[p]$, $a>1$ for a point $p\in Q$; $a[Q]$, $a>1$; and $a\P L_k$ when $a>1$ and $k=n/2$.
This is an example of a multi rigid class, see \cite[Example 1.3]{C14}.
\end{remark}

A related problem is that of \newword{smoothability} of cohomology classes; a cohomology class $\nu$ is said to be smoothable if there exists a smooth subvariety $Y$ with $\nu = [Y]$.
Smoothability has been a subject of study for many authors, and partial results are understood for the Grassmannian.
Less is known for orthogonal Grassmannians.
Rigidity results from \Citeauthor{L25} in \cite{L25} tell us that when a class is rigid and the Schubert variety is singular, then the class is not smoothable.
In \cite{S18}, \Citeauthor{S18} classified smoothness of restriction varieties, which certain classes can be represented by.
Our results give necessary conditions for smoothability in orthogonal Grassmannians.

\begin{comparison}
Thus far, \Cref{question1} has only been studied for multiples of Schubert classes \cite{C14, L25}.
It has been studied for multiples of Schubert classes in other homogeneous varieties, in \cite{B01, C11, C14, C18, CR13, H05, H07, HM13, L25, LSY24a, LSY24b, RT12, W97}. 
A complete classification of multiples of Schubert classes that are realizable over $\Z$ in compact complex Hermitian symmetric spaces is given in \cite{CR13, RT12}.
The question of realizability of cohomology classes in the Grassmannian has been studied most recently by \Citeauthor{CR26} in \cite{CR26} and \Citeauthor{HHMWW25} in \cite{HHMWW25}.
We also build off of \Citeauthor{H12}'s work on realizable classes in products of projective spaces \cite{H12, H14}.
\end{comparison}

\begin{ack}
I was partially supported by an NSF RTG DMS 2037569.
Thank you to my advisor Izzet Coskun for his guidance and support throughout the project. GPT-6 Astra was used during editing and produced minor corrections. 
\end{ack} 


\section{Preliminaries}

\subsection{Notation}\label{sec:notation}
$OG(k,n)$ parameterizes isotropic $k$-planes in an $n$ dimensional complex vector space $V$ with respect to a non-degenerate symmetric bilinear form $Q_0$. 
We can interpret these as projective $(k-1)$-planes on a quadric hypersurface $Q\subset \P^{n-1}$, where $Q$ is cut out by the homogeneous quadratic polynomial associated to $Q_0$. 
When $2k=n$ the paramater space of isotropic $k$-planes has two isomorphic connected components, and we denote by $OG(k,2k)$ one component.

We index Schubert classes in $OG(k,n)$ by sequences $(a_\bullet;b_\bullet)$ 
of length $k$ such that $0< a_1<\cdots<a_s\le \floor*{n/2}$ and 
$\floor*{n/2}-1\ge b_{s+1}>\cdots>b_k\ge 0$
with the property that $a_i\neq b_j+1$ for all $i$ and all $j$. 
Fix an isotropic flag 
\[F_\bullet: F_1\subset F_2\subset\cdots\subset F_{\floor*{n/2}}.\]
In the case that $n$ is even, we must be able to differentiate between the two connected
components of the space of half dimensional linear spaces. We will notate half 
dimensional linear spaces in one component by $F_{n/2}$ and in the other by $F_{n/2-1}^\perp$.
We caution the reader that although the projectivization of the everywhere tangent space
to an isotropic $n/2-2$ plane is the union of two isotropic $\frac{n}{2}-1$ planes, 
$F_{n/2-1}^\perp$ only designates one of the planes: the one that is not 
in the same connected component as $F_{n/2}$. 
The Schubert variety $\Sigma_{a_\bullet;b_\bullet}(F_\bullet)$ is the Zariski closure of 
\[
\Sigma_{a_\bullet;b_\bullet}^\circ(F_\bullet) := \left\{
    [W]\in OG(k,n) \quad \middle\vert \quad
\begin{aligned}
    &\dim(W\cap F_{a_i})=i,\;\forall\; 1\le i\le s,\\
    &\dim(W\cap F_{b_j}^\perp)=j, \;\forall \;s+1\le j\le k.
\end{aligned}
\right\}
\]
Let $\sigma_{a_\bullet;b_\bullet}=[\Sigma_{a_\bullet;b_\bullet}]$ denote its cohomology class. 

Following \cite[Definition 1.1]{C14}, a Schubert class in $OG(k,n)$ is said to be of 
\newword{Grassmannian type} if every flag element in its definition is isotropic and 
none of these flag elements is of dimension $n/2$, i.e. $s=k$ and $a_k<n/2$.
Dually, a Schubert class is of \newword{quadric type} if $s=0$ and $b_1<n/2-1$.
Grassmannian and quadric type classes are dual under the intersection pairing. 

Sometimes we will need to refer to Schubert classes in the Grassmannian $G(k,n)$. 
We index these by sequences $(a_\bullet)$ of length $k$ such that $1\le a_1<\cdots<a_k\le n$ 
and fix an isotropic flag 
\[F_\bullet: F_1\subset F_2\subset\cdots\subset F_{n}.\]
The Schubert variety $\Sigma_{a_\bullet}(F_\bullet)$ is the Zariski closure of 
\[\Sigma_{a_\bullet}^\circ(F_\bullet) := \{[W]\in G(k,n)\;|\; \dim(W\cap F_{a_i})=i, 1\le i\le k\}.\]
Let $\sigma_{a_\bullet}=[\Sigma_{a_\bullet}]$ denote its cohomology class. 
The Schubert classes form an additive basis for the cohomology of $G(k,n)$. 

\subsection{Maps between Grassmannians}\label{sec:maps}

\begin{remark}
The relation to the usual notation for Schubert classes in the Grassmannian given by 
sequences $(\lambda_\bullet)$ where $n-k\ge \lambda_1\ge \lambda_2\ge \ldots\ge \lambda_k\ge 0$ 
is that $a_i=n-k+i-\lambda_i$. 
The advantage of our notation is that it is compatible with maps between Grassmannians. 
That is, given the inclusion $i:G(k,n)\to G(k,n+1)$, we have $i_*\sigma_{a_\bullet}=\sigma_{a_\bullet}$.  

Here, and in the rest of the paper, given a morphism $f:X\to Y$, $f_*$ is a pushforward map in cohomology, called the Gysin map. 
It is defined using the pushforward map in homology and Poincar\'e duality, as seen here 
\[ \begin{tikzcd}
 H^k(X)  \ar[d, "\text{P.D.}"] \ar[r, "f_*", dashed]  & H^{k-(2n-2m)}(Y) \ar[d, "\text{P.D}"] \\
 H_{2n-k}(X) \ar[r, "f_*"] & H_{2n-k}(Y)
\end{tikzcd}\]
where $n=\dim X$ and $m=\dim Y$.
\end{remark}

We also have a natural inclusion map \[i:OG(k,n)\to G(k,n).\]
Given a Schubert class $\sigma_{a_\bullet;}\in H^*(OG(k,n))$ of Grassmannian type,
we have $i_*\sigma_{a_\bullet;}=\sigma_{a_\bullet}\in H^*(G(k,n))$. 

On the other hand, whenever $k\le m\le \floor*{n/2}$, after fixing 
isotropic $F_m$, we have the inclusion \[\phi_{m,n}:G(k,m)\to OG(k,n).\]
Given a Schubert class $\sigma_{a_\bullet}\in H^*(G(k,m))$,
we have $\phi_*\sigma_{a_\bullet}=\sigma_{a_\bullet;}\in H^*(OG(k,n))$. 

There is also a map \[\iota:OG(k,n)\to OG(k+1,n+2)\] defined as follows. 
Let $Q\subset \P^{n+1}$ be a smooth quadric hypersurface and $p$ a point in $Q$.
Then if we take the tangent hyperplane section at $p$ we get a cone with vertex $p$ over a smooth quadric $(n-2)$-fold, $Q'\subset Q$. Given an isotropic $(k-1)$-plane $\Lambda\subset Q'$ we get a unique isotropic $k$-plane in $Q$ by taking $\Lambda'=\overline{p,\Lambda}$. 
Then we have \[\iota:[\Lambda]\mapsto [\overline{p,\Lambda'}].\]
And one can check that \[\iota_*\sigma_{a_\bullet;b_\bullet} = \sigma_{a'_\bullet;b'_\bullet}.\]
where $a'_1=1$, $a'_i=a_{i-1}+1$ for all $2\le i\le s+1$, and $b'_j=b_{j-1}+1$ for all $s+2\le j \le k+1$.

Following the notation in \cite{CR26}, we let 

\begin{align*}
    I_r^R(OG(k,n))&:=\left\{\alpha(a_\bullet,b_\bullet)\in R^k\;\middle| \; \begin{gathered}\sum \alpha(a_\bullet,b_\bullet) \sigma_{(a_\bullet,b_\bullet)}\in H^{2N-2r}(OG(k,n),R) \\ \text{ is realizable over } R \end{gathered} \right\}\\
    I_r^R(G(k,n))&:=\left\{\alpha(a_\bullet)\in R^k\;\middle| \; \sum \alpha(a_\bullet) \sigma_{(a_\bullet)}\in H^{2N-2r}(G(k,n),R) \text{ is realizable over } R  \right\}
\end{align*}
where $R$ is either $\Z$ or $\Q$.


\section{Realizability}

\subsection{Connection to Grassmannian}

We will frequently refer to the following lemma.

\begin{lemma}\label[lemma]{lem:grass-type}
When $2(k+r) < n$ then every class $\nu\in H^{2N-2r}(OG(k,n),\Z)$ is a 
linear combination of classes of Grassmannian type.
By duality, every class $\nu\in H^{2r}(OG(k,n),\Z)$ is a 
linear combination of classes of quadric type.
\end{lemma}

\begin{proof}
By \cite[Proposition 4.15]{C18}, given a Schubert class in $OG(k,n)$, we have
\[\dim \Sigma_{a_\bullet,b_\bullet}=(k-s)(n-k-s-1)-\frac{s(s+1)}{2}+\sum_{i=1}^sa_i+\sum_{j=s+1}^{k}(x_j-b_j)\]
where $x_j:=\#\{i\;|\;a_i\le b_j\}$.
We will first minimize this quantity while having $s<k$. 
For a fixed choice of $k,s,n$, the naive approach might be to make the $a_i$ as small as possible ($a_i=i$)
and the $b_j$ as large as possible ($b_{s+l}=\floor{n/2}-l$). 
Since $\floor{n/2}-(k-s)\ge s$, this will mean that $x_j=s$ for all $j$. 
Now suppose we could do better. The only quantity we can improve on is making some $x_j$ smaller. 
But to decrease $x_j$ by one we would have to increase some $a_i$ by at least two, or decrease $b_j$ by at least two, 
or increase some $a_i$ by one and decrease some $b_j$ by one which would result in a net increase of dimension. 
Hence the initial approach is optimal.
We obtain \[\dim \Sigma_{a_\bullet,b_\bullet}= (k-s)\left(\ceil*{\frac{n}{2}} -\frac{k}{2}-\frac{s}{2}-\frac{1}{2} \right).\]

One can easily verify that with $k,n$ fixed, this is minimized when $s=k-1$. 
Therefore the smallest Schubert variety with $s<k$ is given by 
$\Sigma_{1,2,\ldots,k-1;\floor{n/2}-1}$ which has dimension $\ceil{n/2}-k$.
Hence if $r<n/2-k$ then the class must be of the form $\sigma_{a_\bullet;}$.
However, as we saw in \Cref{sec:notation}, a Grassmannian type Schubert class is one where $s=k$ and $a_k<n/2$. 
Hence in the even dimensional case the smallest non-Grassmannian type Schubert variety is given by $\sigma_{1,2,\ldots, k-2, n/2}$ which has dimension $n/2-k$. 
Thus in $OG(k,n)$, any class with $r<n/2-k$ is of Grassmannian type. 
\end{proof}

We can use this to say when the realizable classes of the orthogonal Grassmannian are exactly the realizable ones in the Grassmannian.

\begin{proposition}\label[proposition]{prop:iff}
Let $\nu$ be a non-zero class in $H^{2N-2r}(OG(k,n),R)$ where $R=\Z$ or $R=\Q$. 
Then if $k+r<\floor{n/2}$, $\nu$ is realizable over $R$ in  
$OG(k,n)$ if and only if $\iota_*\nu\in H^{2N-2r}(G(k,n),R)$ is realizable over $R$ in $G(k,n)$.
\end{proposition}

\begin{proof}
Since $k+r < \floor{n/2}$, by \Cref{lem:grass-type} we know $\nu$ is a combination of classes of Grassmannian type, 
$\nu=\sum \alpha(a_\bullet) \sigma_{a_\bullet}$.
We have maps \[G(k,\floor{n/2})\to OG(k,n)\to G(k,n)\] 
so then 
\[I^R_r(G(k,\floor{n/2}))\subset I^R_r(OG(k,n))\subset I^R_r(G(k,n)).\] 
On the other hand, since $r+k< \floor{n/2}$, by \cite[Proposition 3.8 (1)]{CR26}, 
\[I^R_r(G(k,n))\subset I^R_r(G(k,\floor{n/2})).\] 
Hence $I^R_r(G(k,n))=I^R_r(OG(k,n))$.
\end{proof}

As an immediate application of the proposition and \cite[Theorem 9.1]{CR26} we have

\begin{corollary}
Let $m<\floor{n/2}-2$. Then a nonzero cohomology class 
\[\nu = \smashoperator{\sum_{i=\max(0,m-\floor{n/2}+2)}^{\floor{m/2}}} a_i \sigma_{i+1,m+2-i;} \;\;\;\in H^{2N-2m}(OG(2,n),\Q)\]
is realizable over $\Q$ if and only if the coefficients $a_i$ 
form a log concave sequence of nonnegative rational numbers with no internal zeros.
\end{corollary}

\begin{proof}
\cite[Theorem 9.1]{CR26} says that the statement holds for $G(2,\floor{n/2})$. 
Since $m+2<\floor{n/2}$ the same holds for $G(2,n)$ and we can apply \Cref{prop:iff}
to get it in $OG(2,n)$.
\end{proof}


\subsection{Surfaces}

\begin{theorem}\label{thm:pic2}
Let $\nu=a\sigma_{1,2,\ldots,k-2,k,k+1;}+b\sigma_{1,2,\ldots,k-1;k-1}+c\sigma_{1,2,\ldots,k-2,k;k}$ be a non-zero integral cohomology class in $H^{2N-4}(OG(k,2(k+1)),\Q)$.
If $\nu$ is realizable over $\Q$, then $a,b,c\ge 0$ and $b^2\ge ac$. 
Conversely, if $a,b,c\ge 0$, $a+c\ge b$ and $b^2\ge ac$, then $\nu$ is realizable over $\Q$.
\end{theorem}

We first need the following lemma.

\begin{lemma}\label[lemma]{lem:abcd}
Given $\alpha,\beta,\gamma\in \Q^{\ge 0}$ such that $\beta^2\ge \alpha\gamma$, 
there exist $a,b,c,d\in \Q^{\ge 0}$ such that 
\[a+b=\alpha,\, b+c=\beta,\, c+d=\gamma,\, b^2\ge ac,\, c^2\ge bd\]
if and only if $\alpha+\gamma\ge \beta$. 
\end{lemma}

\begin{proof}
Since \[\alpha+\gamma=a+b+c+d\ge \beta\] the condition is necessary. 
To show it is sufficient, we will show that we can pick 

\[a\in I=\left[ \max \left(0, \alpha-\frac{\beta^2}{\beta+\gamma}\right) , \min\left(\frac{\alpha^2}{\alpha+\beta}, \alpha+\gamma-\beta\right)  \right]\cap \Q.\]
That is, we will show that the set $I$ is nonempty, 
and that if we pick $a\in I$ then we get $a,b,c,d$ satisfying all the desired conditions. 

Since $\alpha+\gamma\ge \beta$, and $\alpha,\beta\ge 0$ 
we know that the upper bounds are nonnegative. 
It is straightforward algebra to verify
\[\beta^2 \ge \alpha\gamma\implies \frac{\alpha^2}{\alpha+\beta} \ge \alpha-\frac{\beta^2}{\beta+\gamma};\]
and
\[\beta^2 \ge \beta^2-\gamma^2\implies  \alpha-\beta+\gamma \ge \alpha - \frac{\beta^2}{\beta+\gamma}.\]

Hence $I$ is nonempty. 
Now let $a\in I$ be arbitrary. 
Then let $b=\alpha-a$.
Since $\alpha\ge \frac{\alpha^2}{\alpha+\beta}\ge a$ 
we know $b\ge 0$. 
Let $c=\beta-\alpha+a$. 
Since $a\ge \alpha-\frac{\beta^2}{\beta+\gamma}\ge  \alpha-\beta$ we have $c\ge 0$. 
Let $d=\alpha+\gamma-\beta-a$.  
Since $a\le \alpha+\gamma-\beta$, then $d\ge 0$. 
It is immediate that $a+b=\alpha$, $b+c=\beta$ and $c+d=\gamma$. 
We just need to check that $b^2\ge ac$ and $c^2\ge bd$. 
Again by straightforward computation we have 
\[\frac{\alpha^2}{\alpha+\beta} \ge a \implies b^2\ge ac;\]
and
\[a \ge \alpha-\frac{\beta^2}{\beta+\gamma}\implies c^2\ge bd.\]
\end{proof}

\begin{proof}[Proof of \Cref{thm:pic2}]
We first show that $b^2\ge ac$ is a necessary condition. 
First observe that there are two divisor classes in $OG(k,2(k+1))$: 
$H_1=\Sigma_{;k,k-2,\ldots,1,0}$ and $H_2=\Sigma_{k+1;k-2,\ldots,1,0}$. 
Note that $Q^{2k}\subset \P^{2k+1}$ has two families of projective $k$-planes.
$H_1$ parameterizes isotropic $(k-1)$-planes meeting a fixed $k$-plane in the 
first family, 
and $H_2$ parameterizes those meeting a fixed $k$-plane in the second family.
Planes from different families are not algebraically equivalent, 
so neither are $H_1$ and $H_2$. 
Hence $\rho(OG(k,2(k+1)))=2$.

We will compute the intersection matrix 
\[M=\begin{bmatrix}
    H_1^2[X] & H_1H_2[X]\\
    H_1H_2[X] & H_2^2[X]
\end{bmatrix}.\] 
The classes $H_1,H_2$ are nef, so by the Hodge Index Theorem, $\det M\le 0$. 

If $\Lambda$ and $\Lambda'$ are two projective $k$-planes in $Q^{2k}$, 
then $\dim \Lambda\cap \Lambda'\equiv k \pmod 2$ 
if and only if $\Lambda$ and $\Lambda'$
belong to the same family.
Hence the parity of $k$ dictates the intersection behavior. 
We will first assume that $k$ is even.
In this case, two general planes from the same family meet in a point, 
while general planes from different families are disjoint. 
If two planes meet in a line they must belong to different families.

We begin by computing $H_1 \cdot \sigma_{1,2,\ldots,k-1;k-1}$.
Taking a general representative of each Schubert class and intersecting them, 
we see that we are looking for the class of 
\[ \{[\Lambda]\in OG(k,2(k+1))\;|\; \Lambda \cap \P F_k^\perp\neq \varnothing, \P G_{k-1}\subset \Lambda\subset \P G_{k-1}^\perp \}.\]
Note that $\P F_k^\perp$ is an isotropic $k$-plane. 
Then $\P F_k^\perp \cap \P G_{k-1}^\perp = L\cong \P^1$ is an isotropic projective line. 
Hence we are looking for $\Lambda$ such that $\P G_{k-1}\subset \Lambda \subset \overline{\P G_{k-1}, L}$.
Note that $\overline{\P G_{k-1}, L}$ is itself an isotropic $k$-plane and meets 
$\P F_k^\perp$ in a line. 
Since $k$ is even, the two $k$ planes belong to different families. 
Hence $H_1\cdot \sigma_{1,2,\ldots,k-1;k-1} = \sigma_{1,2,\ldots,k-1,k+1;}$.

We immediately conclude that $H_1^2 \cdot \sigma_{1,2,\ldots,k-1;k-1} =0$ 
since by the above computation, if there was a $\Lambda$ in the intersection of 
a general representative of $H_1$ and 
$H_1\cdot \sigma_{1,2,\ldots,k-1;k-1}=\sigma_{1,2,\ldots,k-1,k+1;}$,
we would need $\Lambda$ to be contained in a $k$-plane from one family while 
meeting a general $k$-plane in the other. 
    
By symmetry we know $H_2 \cdot \sigma_{1,2,\ldots,k-1;k-1} =\sigma_{1,2,\ldots,k-1;k}$
and $H_2^2  \cdot \sigma_{1,2,\ldots,k-1;k-1} = 0$. 

We now compute $H_1H_2 \cdot \sigma_{1,2,\ldots,k-1;k-1} = H_2\cdot \sigma_{1,2,\ldots,k-1,k+1;}$. 
A representative would correspond to $\Lambda$ such that $\Lambda$ is contained in a $k$-plane 
from one component and meets another $k$-plane in the same component. 
These meet in a point $q$. Furthermore, $\Lambda$ must contain a $\P F_{k-1}$ so 
$\Lambda = \overline{\P F_{k-1},q}$. Hence $H_1H_2 \cdot \sigma_{1,2,\ldots,k-1;k-1}= 1$.

Next, consider $H_1\cdot \sigma_{1,2,\ldots,k-2,k,k+1;}$. 
This consists of $\Lambda$ contained in a $k$-plane from 
one ruling and meeting a $k$-plane from the other ruling. 
If these planes are general they are disjoint so 
$H_1\cdot \sigma_{1,2,\ldots,k-2,k,k+1;}=0$. 
Then $H_1^2\cdot \sigma_{1,2,\ldots,k-2,k,k+1;}$ $=$ $H_1H_2\cdot \sigma_{1,2,\ldots,k-2,k,k+1;}=0$.

Next, $H_2\cdot \sigma_{1,2,\ldots,k-2,k,k+1;}$. 
This consists of $\Lambda$ contained in a $k$-plane 
and meeting a $k$-plane from the same ruling. 
These $k$ planes meet in a point $q$. 
Then $\Lambda$ contains $\overline{\P F_{k-2},q}$, 
so $H_2\cdot \sigma_{1,2,\ldots,k-2,k,k+1;} = \sigma_{1,2,\ldots,k-2,k-1,k+1;}$.

Finally, $H_2^2\cdot \sigma_{1,2,\ldots,k-2,k,k+1;}=H_2\cdot \sigma_{1,2,\ldots,k-2,k-1,k+1;}$. 
This corresponds to $k$ planes containing a $\P F_{k-1}$ 
and the intersection point of the 
two relevant $k$-planes so $H_2^2\cdot \sigma_{1,2,\ldots,k-2,k,k+1;}=1$. 

By symmetry, $H_1H_2\cdot \sigma_{1,2,\ldots,k-2,k;k}=H_2^2\cdot \sigma_{1,2,\ldots,k-2,k;k}=0$
and $H_1^2\cdot \sigma_{1,2,\ldots,k-2,k;k}=1$. 

We obtain the intersection matrix 
\[M_{even} = \begin{bmatrix}
    c&b\\
    b&a
\end{bmatrix}.\]
Then we have $ac-b^2\le 0$, as desired. 

Now assume $k$ is odd. 
In this case, two general planes from the same family are disjoint, 
while general planes from different families meet in a point.  
If two planes meet in a line they must belong to the same family. 
We compute the first entry of $M_{odd}$, and trust that the reader can see how the 
rest of the computation proceeds analogously to the one above, with the only 
difference being in the behavior of the intersection of isotropic $k$-planes. 

The first entry is $H_1^2\cdot [X]$.
We start with $H_1\cdot \sigma_{1,2,\ldots,k-1;k-1}$. 
We need to compute the class of a representative of the intersection of two general representatives:
\[\{[\Lambda]\in OG(k,2(k+1)) \;|\; \Lambda\cap \P F_k^\perp\neq \varnothing, \,\P G_{k-1}\subset \Lambda \subset \P G_{k-1}^\perp \}.\]
Again, $\P F_k^\perp$ is an isotropic $k$ plane, 
and as before $\P F_k^\perp \cap \P G_{k-1}^\perp = L \cong \P^1$, a projective line. 
Hence we get 
$\{[\Lambda]\;|\; \P G_{k-1}\subset \Lambda\subset \overline{\P G_{k-1},L} \}$. 
Now $\overline{\P G_{k-1},L}$ is an isotropic $k$-plane meeting $\P F_k^\perp$ in a line. 
Because $k$ is odd, they belong to the same family. 
Hence $H_1\cdot \sigma_{1,2,\ldots,k-1;k-1} = \sigma_{1,2,\ldots,k-1;k}$.
Then $H_1^2\cdot \sigma_{1,2,\ldots,k-1;k-1}=H_1\cdot \sigma_{1,2,\ldots,k-1;k} = 0$
since $k$-planes from the same family are disjoint. 

Next, consider $H_1\cdot \sigma_{1,2,\ldots,k-2,k,k+1;}$. 
We need to compute the class of 
\[\{[\Lambda]\in OG(k,2(k+1)) \;|\; \Lambda\cap \P F_k^\perp\neq \varnothing, \P G_{k-2}\subset \Lambda\subset \P G_{k+1}\}.\]
$\P F_k^\perp$ and $\P G_{k+1}$ are isotropic $k$-planes from different families so they meet in a point. 
Hence  $H_1\cdot \sigma_{1,2,\ldots,k-2,k,k+1;}=\sigma_{1,2,\ldots,k-1,k+1;}$
and so $H_1^2\cdot \sigma_{1,2,\ldots,k-2,k,k+1;}=H_1\cdot \sigma_{1,2,\ldots,k-1,k+1;} =1$. 

Finally, consider $H_1\cdot \sigma_{1,2,\ldots,k-2,k;k}$. 
We immediately see that this intersection is zero since we require $\Lambda$ to be contained 
in a $k$ plane from one family and meet a different $k$ plane from the same family. 

We conclude that $H_1^2[X]=a$. 
After computing the remaining intersections we get 
\[M_{odd}=\begin{bmatrix}
    a&b\\
    b&c
\end{bmatrix}.\] 
This has the same determinant and hence the same relation as before. 

We now show that $b^2\ge ac$ together with $a+c\ge b$ is a sufficient condition, starting with $OG(2,6)$.
We know that $OG(2,6)\cong F(1,3;4)$. 
Viewing $F(1,3;4)$ as the universal plane, one has $[F(1,3;4)]=h_1+h_2 \in H^2(\P^3\times \P^{3*})$, where $h_1$ is the pullback of the hyperplane class in $\P^3$ and $h_2$ is the pullback of the hyperplane class in $P^{3*}$.
Now consider an irreducible threefold $X\subset \P^3\times \P^{3*}$. 
Then $[X]=ah_1^3+bh_1^2h_2+ch_1h_2^2+dh_2^3$.
By \cite[Theorem 21]{H12} this class is realizable over $\Q$ if and only if $b,c\neq 0$, $b^2\ge ac$ and $c^2\ge bd$. 
Now consider 
\[[X][F(1,3;4)] = (a+b)h_1^3h_2 + (b+c)h_1^2h_2^2 + (c+d)h_1h_2^3.\]
By Bertini's theorem, the class is realizable as an irreducible surface $Y\subset F(1,3;4)\cong OG(2,6)$.
Given in terms of a basis of $H^*(OG(2,6))$ we have  $[Y]=(a+b)\sigma_{23;}+(b+c)\sigma_{1;1}+(c+d)\sigma_{2;2}$.
Then \Cref{lem:abcd} tells us that given a class $\alpha\sigma_{23;}+\beta\sigma_{1;1}+\gamma\sigma_{2;2}\in H^{2N-4}(OG(2,6),\Q)$, if $\beta^2\ge \alpha\gamma$ and $\alpha+\gamma\ge \beta$, we can realize it as a hyperplane section of an irreducible threefold in $\P^3\times \P^{3*}$.

Now that we have the result for $OG(2,6)$, it holds inductively for $OG(k,2(k+1))$. 
We know we have a map $\iota:OG(k,n)\to OG(k+1,n+2)$ as defined in \Cref{sec:maps}. 
Then if we can realize $\nu$ as an irreducible subvariety $X\subset OG(k,2(k+1))$, $\iota(X)$ is an irreducible subvariety in $OG(k+1,2(k+2))$ of class $\iota_*\nu$, which satisfies our assumptions.
\end{proof}


\begin{theorem}\label{thm:surfaces-OG2n}
Let $\nu=a\sigma_{1,4;} + b\sigma_{2,3;}$ be a non-zero cohomology class in $H^{2N-4}(OG(2,n),\Z)$, with $n\ge 9$. 
Then $\nu$ is realizable over $\Z$ if and only if $a>0$ and $b\ge 0$ or $(a,b)=(0,1)$.
\end{theorem}

First we recall the classification of realizable surfaces in the Grassmannian 
from \cite{CR26}.
The statement has been lightly edited to match our notation.

\begin{utheorem}[{c.f. \cite[Theorem 5.1]{CR26}}]
Let $2\le k\le n-k$ and let $\nu\in H^{2N-4}(G(k,n),Z)$ be a 
non-zero, integral cohomology class. 
When $k>2$, every effective class is realizable as an irreducible surface.
When $k=2$ we have $\nu = a\sigma_{1,4} + b\sigma_{2,3}$. 
\begin{enumerate}
    \item When $n>4$, the class $\nu$ can be represented by an
    irreducible surface if and only if $a>0$ and $b\ge 0$ or $(a,b)=(0,1).$
    \item When $(k,n)=(2,4)$, the class $\nu$ can be represented by an irreducible 
    surface if and only if $a,b>0$ or $(a,b)=(1,0)$ or $(a,b)=(0,1)$.
\end{enumerate}
\end{utheorem}

\begin{proof}[Proof of \Cref{thm:surfaces-OG2n}]
When $n\ge 10$ this follows from \Cref{prop:iff} and \cite[Theorem 5.1]{CR26}. 
When $n=9$, the condition $k+r<\floor{n/2}$ fails. 
We know that classes with $a,b>0$, $(a,b)=(1,0)$ or $(a,b)=(0,1)$ are realizable in $G(2,4)$ and hence in $OG(2,9)$. 
The class $b\sigma_{2,3}$ is not realizable in $G(2,9)$ for $b>1$, hence not realizable in $OG(2,9)$.
It remains to show that $a\sigma_{1,4}$ is realizable in $OG(2,9)$ for all $a\ge 1$.
$OG(2,9)$ is the parameter space of lines on a quadric 7-fold $Q\subset \P^8$. 
We take a tangent hyperplane section and get a conic with vertex the point $p$ over $Q'$, a smooth quadric 5-fold sitting in $Q$. 
Observe that for all $a\ge 1$ there is an irreducible degree $a$ surface $X\subset Q'$.
One way to get this is to take an irreducible degree $a$ curve in a quadric surface in $Q'$ and take a cone in $Q'$.
Then $X\subset Q'\cong OG(1,7)$ and $[X]=a\sigma_{3;}$. 
Hence $\iota(X)\subset OG(2,9)$ and $[\iota(X)]=\iota_*[X]=a\sigma_{14;}$, where again $\iota$ is the map defined in \Cref{sec:notation}.
Then $\nu$ is realizable in $OG(2,9)$ if and only if it is realizable in $G(2,9)$.
\end{proof}


\subsection{Codimension 2}

We recall the relevant result from \cite{CR26}, 
adjusting the notation. 

\begin{utheorem}[{c.f. \cite[Theorem 5.2]{CR26}}]
Let $2\le k\le n-k$ and let $\nu\in H^4(G(k,n),\Z)$ be a non-zero,
integral cohomology class. 
When $k>2$, every effective class is realizable. 
When $k=2$ we have $\nu = a\sigma_{n-3,n}+b\sigma_{n-2,n-1}$.
In this case, if additionally $n>4$, 
the class $\nu\in H^4(G(2,n),\Z)$ is realizable over $\Z$ 
if and only if $a>0$ and $b\ge 0$ or $(a,b)=(0,1).$
\end{utheorem}

\begin{theorem}\label{thm:codim2}
Let $\nu =a\sigma_{;k+1,k-2,k-3,\ldots,1,0}+b\sigma_{;k,k-1,k-3,\ldots,1,0}$ be a nonzero integral cohomology class in $H^4(OG(k,n),\Z)$.
If $n>2k+4$, then:
\begin{enumerate}
    \item When $k>2$, $\nu$ is realizable over $\Z$ if and only if $a,b\ge 0$.
    \item When $k=2$, $\nu$ is realizable over $\Z$ if and only if $a>0$, $b\ge 0$ or $(a,b)=(0,1)$.
\end{enumerate}
\end{theorem}

\begin{proof}
Since $n>2k+4$, by \Cref{lem:grass-type}, 
the only codimension 2 Schubert varieties must be of quadric type. 
These are $\sigma_{;k+1,k-2,k-3,\ldots,1,0}$ and $\sigma_{;k,k-1,k-3,\ldots,1,0}$.

Let $\iota:OG(k,n)\to G(k,n)$ be the natural inclusion. 
The codimension 2 classes in $G(k,n)$ are $\sigma_{n-k-1,n-k+2,n-k+3,\ldots,n}$ 
and $\sigma_{n-k,n-k+1,n-k+3,\ldots,n}$ 
We will show that 
\[\iota^*\sigma_{n-k-1,n-k+2,n-k+3,\ldots,n}=\sigma_{;k+1,k-2,k-3,\ldots,1,0}\]
and 
\[\iota^*\sigma_{n-k,n-k+1,n-k+3,\ldots,n}=\sigma_{;k,k-1,k-3,\ldots,1,0}.\] 
Because $\iota^*$ respects grading by codimension,
 \[\iota^*\sigma_{n-k-1,n-k+2,n-k+3,\ldots,n}=\alpha\sigma_{;k+1,k-2,k-3,\ldots,1,0}+\beta\sigma_{;k,k-1,k-3,\ldots,1,0}\]
for some $\alpha,\beta\in \Z$. 

The dual of $\alpha\sigma_{;k+1,k-2,k-3,\ldots,1,0}$ is the Grassmannian type class $\sigma_{1,2,\ldots, k-1,k+2;}$. 
Then by the push-pull formula, we get that $\alpha=1$ and $\beta=0$. 
The computation for $\iota^*\sigma_{n-k,n-k+1,n-k+3,\ldots,n}$ is analogous.

Now consider the class $\nu=a\sigma_{n-k-1,n-k+2,n-k+3,\ldots,n}+b\sigma_{n-k,n-k+1,n-k+3,\ldots,n}$
in $H^4(G(k,n),\Z)$. 
By \cite{CR26}, if $k>2$ and $a,b\ge 0$, $\nu$ is realizable 
as an irreducible subvariety of $G(k,n)$.
Let $Y$ be a general irreducible representative. 
We know $[OG(k,n)]=2^k\sigma_{k,k-1,\ldots,1}\in H^*(G(k,n),\Z)$ 
(in the usual Grassmannian notation).
Then, using Coskun's Littlewood-Richardson rule \cite{C09},
one can check that $[Y]\cdot \sigma_{k+1,k,\ldots,2}\neq 0$.
Hence by \cite[Theorem 8.1]{D96}, $Y\cap OG(k,n)$ is irreducible.
We have 
\[[Y\cap OG(k,n)]=\iota^*(\nu)=a\sigma_{;k+1,k-2,k-3,\ldots,1,0}+b\sigma_{;k,k-1,k-3,\ldots,1,0}.\]
Hence any such class is realizable in $OG(k,n)$. 

If $k=2$, $n\ge 9$, $a>0, b\ge 0$ or $(a,b)=(0,1)$,
then $\nu$ is realizable in $G(2,n)$. 
By the same argument as above, 
$\iota^*\nu = a\sigma_{;3,0}+b\sigma_{;2,1}$ 
is realizable in $OG(2,n)$, $n\ge 9$. 

We claim that $b\sigma_{;2,1}$, is realizable in $OG(2,n)$, 
$n\ge 9$, if and only if $b=1$.
Suppose \[[X]=b\sigma_{;2,1}\in H^4(OG(2,n),\Z), \]
where $X$ is an irreducible subvariety. 
Then $\dim X=2n-9$. 
Let 
\[Z =\bigcup_{[\Lambda]\in X} \Lambda \subset \P^{n-1}\]
be the variety swept out by $X$.
Since $\sigma_{2,1}\cdot \sigma_{1;1}=0$, 
$Z$ does not contain a general point,
so $Z$ cannot sweep out the quadric $Q$. 
Hence $\dim Z\le n-3$.
On the other hand we have
\[2n-9 = \dim X \le \dim F_1(Z)\le 2\dim Z -2\]
so $\dim Z=n-3$.

To complete the proof we need
\begin{lemma}[{c.f. \cite{S48}}]
Suppose $Z\subset \P^N$ is a projective variety of dimension $d$.
If the Fano variety of lines on $Z$, $F_1(Z)$ has dimension $2d-3$ then either $Z$ is swept out 
by a 1-parameter family of linear $\P^{d-1}$s 
or $Z$ is a quadric hypersurface in $\P^{d+1}$.  
\end{lemma}

Here we have $2(n-3)-3=2n-9$. 
But $Q$ does not contain any linear $\P^{n-4}$ so $Z$ must be a quadric. 
We know $[F_1(Z)]=\sigma_{;2,1}$, so $b=1$.
\end{proof}

\newpage

\printbibliography

\end{document}